\documentclass[11pt, reqno]{amsart}
\usepackage{graphicx}              
\usepackage{amsmath}              
\usepackage{amsfonts}              
\usepackage{amsthm}                
\usepackage{color}
\usepackage{tabulary}
\usepackage{caption}
\usepackage{subcaption}
\usepackage{amssymb}
\usepackage{todonotes}
\theoremstyle{plain}
\usepackage{mathtools}
\usepackage{soul}

\theoremstyle{plain}

\newtheorem{theorem}{Theorem}[section]

\newtheorem{proposition}[theorem]{Proposition}

\newtheorem{axiom}[theorem]{Axiom}

\theoremstyle{definition}

\newcommand{\nc}{\newcommand}
\nc{\dmo}{\DeclareMathOperator}

\nc{\Q}{\mathbb{Q}}
\nc{\R}{\mathbb{R}}
\nc{\Z}{\mathbb{Z}}
\nc{\N}{\mathbb{N}}
\nc{\C}{\mathbb{C}}
\nc{\cS}{\mathcal{S}}
\nc{\iso}{\cong}
\dmo{\Mod}{Mod}
\dmo{\Diff}{Diff}
\dmo{\Homeo}{Homeo}
\dmo{\dist}{dist}
\dmo\BDiff{BDiff}
\dmo\SO{SO}
\dmo\slide{sl}
\dmo\im{im}
\dmo\id{id}
\dmo\Fix{Fix}
\dmo\Out{Out}
\dmo{\T}{\mathcal{T}}
\dmo{\Te}{\mathcal{T}^{\epsilon}}
\dmo{\M}{\mathcal{M}}
\dmo{\Me}{\mathcal{M}^{\epsilon}}

\renewcommand{\epsilon}{\varepsilon}
\nc{\coloneq}{\mathrel{\mathop:}\mkern-1.2mu=}
\nc{\margin}[1]{\marginpar{\scriptsize #1}}
\nc{\para}[1]{\bigskip\noindent\textbf{#1}}

\title [Equidistribution along random walks]{Equidistribution of closed geodesics along random walk trajectories with respect to the harmonic invariant measure}
\author{Ilya Gekhtman}

\begin{document}
\begin{abstract}
We prove that for suitable random walks on isometry groups of $CAT(-1)$ spaces, typical sample paths eventually land on loxodromic elements which equidistribute with respect to a flow invariant measure on the unit tangent bundle of the quotient space.
\end{abstract}
\maketitle

\section{Introduction and statement of results}

Let $(X,d)$ be a $CAT(-1)$ space and $G<Isom(X)$ a nonelementary discrete subgroup.
Let $M=X/G$.
Let $$T^{1}X=\partial^{2}X\times \mathbb{R}$$ and 
$$T^{1}M=T^{1}X/G$$ the unit tangent bundles of $X$ and $m$ respectively.
Let $\pi:T^{1}X\to X$ be the canonical projection.
Let $\tilde{g_{t}}$ be the geodesic flow on $T^{1}X$ and $g_{t}$ the geodesic flow on $T^{1}M$. 
Let $\tilde{d}_{T}$ be an $Isom(X)$ invariant metric on $T^{1}X$ satisfying 
  $$c^{-1}d(x,y)\leq \inf_{p\in \pi^{-1}x,q\in \pi^{-1}y}\tilde{d}_{T}(p,q)\leq c d(x,y).$$
For $x\in X$ and $\zeta \in X\cup \partial X$ let $\gamma_{x,\zeta}$ be the unit speed geodesic ray from $x$ in the direction of $\zeta$, considered as a subset of $T^{1}X$.  
The translation length of $g\in G$ is defined as
$l(g)=\inf_{x\in X}d(x,gx).$
If $l(g)>0$ then the infimum is realized and $g$ is called a loxodromic isometry of $X$. 
The set of points in $X$ realizing the infimum is a geodesic $s_{g}$ in $X$, and can be naturally considered a subset of $T^{1}X$.
Let $\gamma=\gamma_{g}$ be the associated unit speed closed geodesic on $M=X/G$ and $D_{\gamma}=D_{g}$ the arclength Lebesgue measure on $\gamma_{g}$, considered as a measure on $T^{1}M$ normalized to unit mass. 
The Bowen-Margulis measure $m$ on $T^{1}M$ is the geodesic flow invariant measure corresponding to the Patterson-Sullivan geodesic current on $\partial X \times \partial X$ (see \cite{Roblin} for details).

Roblin \cite{Roblin} proved the following equidistribution result for  closed geodesics on $M$ with respect to $m$. 
More precisely, assume
that $m$ is finite, and normalized to have unit mass. Let $L_{R}$ be the set of closed geodesics on $M$ with length at most $R$.

Then the measures on $T^{1}M$ defined by
$$\mathcal{L}_{R}=\frac{1}{hR}\sum_{\gamma \in L_{R}}D_{\gamma}$$ weakly converge to $m$ as $R\to \infty$.

In this note, we prove a different sort of equidistribution result: along  typical trajectories of random walks on $G$, closed geodesics associated to loxodromic elements equidistribute with respect to the "harmonic invariant measure."

Let $\mu$ be a probability measure on $G$ whose support generates $G$ as a semigroup.
Let $\hat{mu}(g)=\mu(g^{-1})$ be the reflected measure of $\mu$.

Let $\mu^{\mathbb{N}}$ be the product measure on $G^{\mathbb{N}}$.
Let $T:G^{\mathbb{N}}\to G^{\mathbb{N}}$ be the transformation that
takes the one-sided infinite sequence $(h_{i})_{i\in \mathbb{N}}$
to the sequence $(\omega_{i})_{i\in \mathbb{N}}$ with
$$\omega_{n}=h_{1}\cdots  h_{n}.$$

Let $P$ be the pushforward measure
$T_{*}\mu^{\mathbb{N}}$.

The measure $P$ describes the distribution of $\mu$
sample paths, i.e. of products of independent $\mu$-distributed increments.

The following results were proved by Maher-Tiozzo \cite{MT} in the more general setting of geodesic Gromov hyperbolic spaces.
\begin{proposition}
For any $x\in X$ and $P$ a.e.
sample path $\omega=(\omega_{n})_{n\in \mathbb{N}}$ of the random walk
on $(G, \mu)$, the sequence $(\omega_{n}x_0)_{n\in \mathbb{N}}$ converges
to a point $\omega_{+}=bnd \omega \in \partial X$. 
If in addition $\mu$ has finite first moment with respect to the metric
$(X,d)$, then there exists $L>0$ such that
for $P$-a.e. sample path $\omega$ {and for every $x_0\in X$} one has
$$\lim_{n\to \infty}\frac{d_{X}(x_{0},\omega_{n}x_{0})}{n}=L$$
The measure $\nu=bnd_{*}P$ is the unique $G$ stationary measure on $\partial X$.
\end{proposition}

Moreover, Tiozzo proved that sample paths sublinearly track geodesics.
\begin{proposition}\label{sublinear}
\cite{Tiozzo}
Assume $\mu$ has finite first moment with respect to $(X,d)$. 
For $P$-a.e. sample path $\omega$ and every $x\in X$ one has
$$\frac{d(\omega_{n}x,\gamma_{x,\omega_{+}})}{n}\to 0.$$
\end{proposition}
Moreover, Dahmani-Horbez \cite{Dahmani-Horbez} proved (again in the more general setting of Gromov hyperbolic spaces) that a typical sample path eventually lands on loxodromic elements which pass close to the basepoint.

\begin{proposition}
Let $G\curvearrowright X$ be a nonelementary discrete action on a $CAT(-1)$ space and $\mu$ a  measure with finite second moment on $G$ whose support generates $G$ as a semigroup.
Then: 

a)For $P$-a.e. $\omega \in G^{\mathbb{N}}$ $$L=\lim_{n\to \infty}\frac{l(\omega_{n})}{n}.$$  In particular $\omega_{n}$ is loxodromic for large enough $n$. 

b) For any $c>0$, for all $\epsilon>0$ and for $P$-a.e. sample path $\omega$
there is an $n_{0}\in \mathbb{N}$
such that for all $n\geq n_0$, $\gamma(\epsilon Ln, (1-\epsilon)Ln)$ is contained in the $c$ neighborhood in $T^{1}X$ of the
axis of $\omega_n$ where $\gamma$ is the unit speed geodesic from a basepoint $o\in X$ to $\omega_{+}\in \partial X$.
\end{proposition}
\begin{proof}

a)This is a special case of \cite[Theorem 0.2]{Dahmani-Horbez}

b)The corresponding statement for some $c>0$ is  a special case of \cite[Proposition 1.9]{Dahmani-Horbez}.
We then can obtain it for  arbitrarily small $c>0$ by using the following fact, which follows from comparison geometry: for any $c_{1}>c_{2}>0$ there is a $t_{0}>0$ such that whenever $\gamma_{1}$ and $\gamma_{2}$ are parametrized unit speed geodesics in a $CAT(-1)$
space such that $d(\gamma_{1}(t),\gamma_{2}(t))\leq c_{1}$ for all $|t|<N$ we have $d(\gamma_{1}(t),\gamma_{2}(t))\leq c_{2}$ for all $|t|<N-t_{0}$ 
\end{proof}

In this note, we are interested in the asymptotic behavior, as $n\to \infty$ of the closed geodesic corresponding to the loxodromic element $\omega_n$ along typical sample paths of the random walk.
In particular, we will show they equidistribute with respect to a certain measure on the unit tangent bundle of $M$. 

Let $\nu$ and $\hat{\nu}$ be respectively the $\mu$ and $\hat{\mu}$ stationary probability measures on $\partial X$ respectively.

The action $G\curvearrowright X\times X$ preserves the measure class of $\nu \times \nu$ and is ergodic with respect to it \cite{Kaistrip}.

We will make the following assumption about $G$ and $\mu$.
\begin{axiom}\label{harmonicinvariant}

There exists a $G$ invariant and geodesic flow invariant Radon measure $\tilde{m}$ on the unit tangent bundle 
$$T^{1}X=\partial^{2}X\times \mathbb{R},$$ in the measure class of $\hat{\nu} \times \nu \times Leb$ (where $Leb$ is the Lebesgue measure on $\mathbb{R}$) which projects to a probability measure $m$ on the unit tangent bundle $T^{1}M=T^{1}X/G$ of $M$. We call $m$ the harmonic invariant measure of $(G,\mu)$.
\end{axiom}

Axiom \ref{harmonicinvariant} is known to hold in the following cases.
\begin{itemize}
\item When $\mu$ is the discretization of Brownian motion on the universal cover of a compact manifold of negative curvature (Anderson-Schoen \cite{AS}).
\item When $G<Isom(X)$ is convex cocompact and $
\mu$ has finite support.
((Kaimanovich \cite[Theorem 3.1]{Kaierg}).
\item When $G<Isom(X)$ is geometrically finite and $\mu$ has finite support (Gekhtman-Gerasimov-Potyagailo-Yang \cite[Theorem 1.8]{GGPY}.) In this setting, when $G$ contains parabolic elements \cite[Theorem 1.5]{GGPY} implies that the harmonic invariant measure is singular to the Bowen-Margulis measure.
\end{itemize}

 Recall a family $\sigma_{n}$ of Borel probability measures on a metrizable space $Z$ weakly converges to a measure $\sigma$ on $Z$, if for every bounded continuous function $g$ on $Z$, $$\int g d\sigma_{n}\to \int g d\sigma.$$
By the Portmanteau theorem, this is equivalent to: for any Borel $A\subset Z$,
$$\lim \sup \sigma_{n}(A)\leq \sigma(\overline{A})$$
and to: for any Borel $A\subset Z$,
$$\lim \inf \sigma_{n}(A)\geq \sigma(int A)$$ where $int A$ is the interior of $A$.

The goal of this note is to prove the following. \begin{theorem}\label{randomequid}
Suppose $X$ is a proper $CAT(-1)$ space, $G<Isom(X)$ is a nonelementary discrete subgroup, and $\mu$ a measure on $G$ with finite second moment with respect to $(X,d)$ and satisfying Axiom \ref{harmonicinvariant}.
Then for $P$-almost every sample path $\omega\in G^{\mathbb{N}}$  the measures $D_{w_n}$
weakly converge to the harmonic invariant measure $m$ as $n\to \infty$.
\end{theorem}
\section{Proof of Theorem 
\ref{harmonicinvariant}}

For a Borel subset $A\subset T^{1}X/G$ let $\tilde{A}$ be the $G$ invariant preimage of $A$ in $T^{1}X$.
Let $N_{r}A$ and $I_{r}A$ the $r$ neighborhood and $r$ interior of $A$ respectively, with respect to the metric $d_{T}$.
The measure $\hat{\nu} \times \nu$ is known to be ergodic with respect to the $G$ action, and hence $m$ is ergodic
with respect to the geodesic flow (see \cite{Kaierg}). 

Thus the Birkhoff ergodic theorem
implies
$$\frac{1}{T}|\{t\in [0,T]:g_{t}q\in A\}|\to m(A)$$ for $m$-a.e. $q\in T^{1}M$.

Consequently,

$$\frac{1}{T}|\{t\in [0,T]:\tilde{g}_{t}q\in \tilde{A}\}|\to m(A)$$ for $\tilde{m}$-a.e. $q\in T^{1}X$.
Consequently, since geodesics converging to the same boundary point are asymptotic, we have that for any $r>0$, 
for $P$ a.e. $\omega \in G^{\mathbb{N}}$, any limit point as $T\to \infty$ of 

$$\frac{1}{T}|\{t\in [0,T]:\gamma_{o,\omega_{+}}(t)\in \tilde{A}\}|$$ is bounded between $m(I_{r}A)$ and $m(N_{r}A)$.
In particular for large $n$ we have 
$$m(I_{2r}A)\leq \frac{1}{Ln}|\{t\in [0,Ln]:\gamma_{o,\omega_{+}}(t)\in \tilde{A}\}|\leq m(N_{2r}A)$$ and thus
$$m(I_{2r}A)-2r\leq \frac{1}{Ln}|\{t\in [rLn,(1-r)Ln]:\gamma_{o,\omega_{+}}(t)\in \tilde{A}\}|\leq m(N_{2r}A)$$

By Proposition \ref{sublinear}, $d(\omega_{n}o,\gamma_{o,\omega_{+}}(Ln))/n\to 0$, so since $X$ is $CAT(-1)$ we have for all large enough $n$:

$$d(\gamma_{o,\omega_{+}}(t), \gamma_{o,\omega_{n}}(t))\leq r$$ for all $0\leq t\leq (1-r)Ln.$
Thus, 
$$m(I_{3r}A)-2r\leq \frac{1}{Ln}|\{t\in [rLn,(1-r)Ln]:\gamma_{o,\omega_{+}}(t)\in \tilde{A}\}|\leq m(N_{3r}A)$$

Thus, for a suitable unit speed parametrization of the axis $s=s_{\omega_n}$ of $\omega_n$, Proposition \ref{randomequid} implies for large enough $n$:

$$m(I_{3r}A)-3r\leq \frac{1}{Ln}|\{t\in [rLn+r,(1-r)Ln-r]:s(t)\in \tilde{A}\}|\leq m(N_{3r}A)$$
and thus 
$$m(I_{3r}A)-3r\leq \frac{1}{Ln}|\{t\in [0,Ln]:s(t)\in \tilde{A}\}|\leq m(N_{3r}A)+3r.$$

Since $l(\omega_{n})/n\to L$ this implies for large enough $n$:

$$m(I_{3r}A)-3r\leq \frac{1}{l(\omega_{n})}|\gamma_{\omega_{n}}\cap A|\leq m(N_{3r}A)+3r.$$

Letting $r\to 0$ gives 

$$m(int(A))\leq \lim \inf_{n\to \infty} \frac{1}{l(\omega_{n})}|\gamma_{\omega_{n}}\cap A|$$ 
and 
$$\lim \sup_{n\to \infty} \frac{1}{l(\omega_{n})}|\gamma_{\omega_{n}}\cap A|\leq m(\overline{A})$$ 
completing the proof.

\section{Acknowledgments}
I would like to thank Ursula Hamenstaedt, Vadim Kaimanovich, and Francois Leddrapier for very useful conversations.

\end{document}